\documentclass{amsart}
\usepackage{amssymb}
\usepackage{amsmath}
\usepackage{amsthm}
\usepackage{latexsym}
\usepackage{amsfonts}
\usepackage{mathrsfs}
\usepackage{stmaryrd}
\usepackage{relsize}

\newtheorem{thm}{Theorem}[section]
\newtheorem{cor}[thm]{Corollary}
\newtheorem{lem}[thm]{Lemma}
\newtheorem{prop}[thm]{Proposition}
\theoremstyle{definition}
\newtheorem{defn}[thm]{Definition}
\theoremstyle{remark}
\newtheorem{rem}[thm]{Remark}
\newtheorem{ex}{Example}
\numberwithin{equation}{section}

\renewcommand{\emptyset}{\varnothing}

\DeclareMathOperator{\p}{p} \DeclareMathOperator{\ap}{ap} \DeclareMathOperator{\ess}{ess}
\DeclareMathOperator{\cc}{c} \DeclareMathOperator{\rr}{r}
\newcommand{\R}{\ensuremath{\mathbb R}}
\newcommand{\C}{\ensuremath{\mathbb C}}
\newcommand{\N}{\ensuremath{\mathbb N}}

\newcommand{\Comp}{\ensuremath{\mathbb C}}

\newcommand{\product}{[\cdot\,,\cdot]}

\newcommand{\calK}{\mathcal K}

\newcommand{\calR}{\mathcal R}

\newcommand{\calU}{\mathcal U}

\newcommand{\calX}{\mathcal X}         
\newcommand{\calY}{\mathcal Y}

\newcommand{\la}{\lambda}

\newcommand{\ol}{\overline}

\renewcommand{\Im}{\operatorname{Im}}

\newcommand{\dist}{\operatorname{dist}}
\renewcommand{\ker}{\operatorname{ker}}
\newcommand{\ran}{\operatorname{ran}}

\newcommand{\dom}{\operatorname{dom}}

\newcommand{\Llra}{\Longleftrightarrow}

\newcommand{\wto}{\rightharpoonup}

\newcommand{\adj}{{\mathsmaller[*\mathsmaller]}}
\newcommand{\sap}{\sigma_{\ap}}
\newcommand{\spp}{\sigma_{++}}
\newcommand{\spip}{\sigma_{\pi_+}}
\newcommand{\spim}{\sigma_{\pi_-}}
\newcommand{\smm}{\sigma_{--}}

\begin{document}

%-------------------------------------------------------------------------
% editorial commands: to be inserted by the editorial office
%
%\firstpage{1} \volume{228} \Copyrightyear{2004} \DOI{003-0001}
%\seriesextraline{This is the Concrete Title of this Book\br H.E. R and S.T.C. W, Eds.}
%\firstpage{1}

\title[Local Definitizability of $T^\adj T$ and $TT^\adj$]
{Local Definitizability of $T^\adj T$ and $TT^\adj$}

%----------Author 1
\author[Philipp]{Friedrich Philipp}
\address{
Technische Universit\"{a}t Ilmenau\\
Institut f\"{u}r Mathematik\\
Postfach 100565\\
D-98684 Ilmenau\\
Germany} \email{friedrich.philipp@tu-ilmenau.de}

%----------Author 2
\author[Ran]{Andr\'e C.M.\ Ran}
\address{
Afdeling Wiskunde\\
Faculteit der Exacte Wetenschappen\\
Vrije Universiteit Amsterdam\\
De Boelelaan 1081a\\
1081 HV Amsterdam\\
The Netherlands } \email{ran@few.vu.nl}

%----------Author 3
\author[Wojtylak]{Micha\l{} Wojtylak}
\address{
Instytut Matematyki\\
Uniwersytet Jagiello\'{n}ski\\
ul.\ \L{}ojasiewicza 6\\
30--348 Krak\'{o}w\\
Poland } \email{michal.wojtylak@gmail.com}



\date{December 18, 2009}

\thanks{The first author gratefully acknowledges support from the
Deutsche For\-schungs\-ge\-meinschaft (DFG), grant BE 3765/5-1 TR 904/4-1. The third author was supported by
the EU Sixth Framework Programme for the Transfer of Knowledge ``Operator theory methods for differential
equations'' (TODEQ) \# MTKD-CT-2005-030042, he also thanks the Faculty of Sciences of the VU University
Amsterdam, where the research was partially carried out during his Post-Doc stay.}

\begin{abstract}
The spectral properties of two products $AB$ and $BA$ of possibly unbounded operators $A$ and $B$ in a Banach
space are considered. The results are applied in the comparison of local spectral properties of the operators
$T^\adj T$ and $TT^\adj$ in a Krein space.  It is shown that under the assumption that both operators $T^\adj
T$ and $TT^\adj$ have  non-empty resolvent sets, the operator $T^\adj T$ is locally definitizable if and only
if $TT^\adj$ is. In this context the critical points of both operators are compared.
\end{abstract}

%%% ----------------------------------------------------------------------
\maketitle
%%% ----------------------------------------------------------------------
%\tableofcontents

%%%%%%%%%%%%%%%%%%%%%%%%%%%%%%%%%%%%%%%%%%%%%%%%%%%%%%%%%%%%%%%%%%%%%%%%%%
%%%%%%%%%%%%%%%%%%%%%%   HERE WE GO   %%%%%%%%%%%%%%%%%%%%%%%%%%%%%%%%%%%%
%%%%%%%%%%%%%%%%%%%%%%%%%%%%%%%%%%%%%%%%%%%%%%%%%%%%%%%%%%%%%%%%%%%%%%%%%%
\section*{Introduction}
In the paper \cite{rw} three conditions on a closed and densely defined operator $T$ in a Krein space $\calK$
were considered:
\begin{itemize}
\item[(t1)] $T^\adj T$ and $TT^\adj$ are selfadjoint operators in $\calK$;
\item[(t2)] $T^\adj T$ and $TT^\adj$ have non-empty resolvent sets;
\item[(t3)] $T^\adj T$ is definitizable.
\end{itemize}
Under these conditions the spectral properties of the operators $T^\adj T$ and $TT^\adj$ were compaired. In
particular, it was shown, that if (t1)--(t3) are satisfied, then the non-zero spectra, as well as the
non-zero singular and the non-zero regular critical points, respectively, of $T^\adj T$ and $TT^\adj$
coincide. On the other hand, there were given counterexamples, showing that for the point zero the same will
not be true.

The present contribution can be regarded as a continuation of the paper \cite{rw}, although it also contains
some more general results which we consider to be of independent interest. The body of the paper consists of
two sections. In the first one we consider the spectra of the products $AB$ and $BA$ of two arbitrary linear
operators $A$ and $B$ acting between Banach spaces. We give a simple proof of a theorem from \cite{hkm},
saying that the non-zero spectra of $AB$ and $BA$ are equal, provided the resolvent sets of $AB$ and $BA$ are
non-empty. Moreover, we show that not only the non-zero spectra of $AB$ and $BA$ coincide, but also the most
prevalent types of spectra. At the same time we establish some mutual estimates on the norms of the
resolvents of $AB$ and $BA$.

The second part of the paper is devoted entirely to the situation when $A=T$ is a closed, densely defined
operator in a Krein space and $B=T^\adj$ is its Krein space adjoint. There are two main goals. The first one
is showing that (t2) implies (t1)  and is accomplished in Theorem \ref{t:rho->sa}. Our main tool is here a
Banach-space result from \cite{hm}. The second goal is to prove analogues of central results of \cite{rw}
assuming -- instead of (t3) -- that the operator $T^\adj T$ is definitizable only over a set $\Omega$ (see
Definition \ref{d:locdef}). First, we provide a natural correspondence between the sign types of the spectra
of $T^\adj T$ and $TT^\adj$ (Proposition \ref{p:spectra}). Later on, this fact is used in the proof of
Theorem \ref{t:locdef}. This theorem states that under condition (t2) the operator $T^\adj T$ is
definitizable over a set $\Omega$  if and only if $TT^\adj$ is. This was proved already in \cite{rw} for
$\Omega=\ol\Comp$, since  definitizability over $\ol\Comp$ is equivalent to definitizability. However, in the
present situation we cannot use the technique of definitizing polynomials as in \cite{rw}. Instead, we have
to tackle the problem by comparing the local sign type properties of the spectra of $T^\adj T$ and $TT^\adj$.
In this setting we also prove the equality of the sets on non-zero critical points of $T^\adj T$ and
$TT^\adj$ (Theorem \ref{critical}), which also has its analogue in \cite{rw}.  The following simple example
shows that all these generalizations are substantial, i.e. locally definitizable but not definitizable
operators of the form $T^\adj T$ do exist.

%(see Definition \ref{d:locdef}). %Roughly speaking, a selfadjoint operator in $\calK$
%is locally definitizable over a domain $\Omega\subset\ol\C$ if it can be decomposed into a $J$-orthogonal sum
%of a definitizable operator and an operator with spectrum outside of $\Omega$.
%The key result,  Theorem \ref{t:locdef}, says that under condition (t2) the operator $T^\adj
%T$ is locally definitizable over a set $\Omega$ if and only if $TT^\adj$ is.
%This fact was proved already in  \cite{rw} for $\Omega=\bar\Comp$, since local definitizability over
%$\bar\Comp$ is equvalent to definitizability.
 % Furthermore,  we also prove a natural
%correspondence between the sign types of the spectra  (Proposition \ref{p:spectra}) and critical points of
%$T^\adj T$ and $TT^\adj$  (Theorem \ref{critical}). These both results are generalisations  of the theorems
%from \cite{rw} to the locally definitizable case. The following simple example shows that this generalization
%is substantial.

\begin{ex}\label{e:one}
Let $(T_n)_{n=0}^\infty$  be a bounded sequences of linear operators in $\C^2$, and let the fundamental
symmetry $J_0=\begin{pmatrix} 0 & 1\\1 & 0\end{pmatrix}$ determine the indefinite inner product on $\C^2$.
Suppose additionally, that for each $n\in\N$ the operator $T_n^\adj T_n$ (and thus also $T_n T_n^\adj$) has
exactly one (real) eigenvalue $\lambda_n$ and that the sequence $(\lambda_n)_{n=1}^\infty$ is strictly
decreasing to zero\footnote{E.g. $T_n=\begin{pmatrix} 1/n & 0\\0 & 1/n\end{pmatrix}U_n$ with any $U_n$
satisfying $U_n^\adj=U_n^{-1}$ }. In the space $\ell^2(\C^2)$ we consider the operator $T$ and the
fundamental symmetry $J$, defined by
$$
T = \bigoplus_{n=1}^\infty T_n,\qquad  J=\bigoplus_{n=1}^\infty J_0.
$$
Then the operators $T^\adj T$ and $TT^\adj$ are given by
$$
T^\adj T = \bigoplus_{n=1}^\infty\,T_n^\adj T_n \quad\text{ and }\quad TT^\adj = \bigoplus_{n=1}^\infty\,T_n
T_n^\adj
$$
and they satisfy (t1) and (t2) as bounded operators. Note that the algebraic eigenspace of each of the
operators  $T^\adj T$ and $TT^\adj$ corresponding to the eigenvalue $\lambda_n$ ($n\in\N$) is two-dimensional
and indefinite. Therefore, both operators are (locally) definitizable over $\ol\C\setminus\{0\}$, but not
definitizable (over $\ol\C$).
\end{ex}

For a history of the problem of comparing the operators $T^\adj T$ and $TT^\adj$ and its relation with
indefinite polar decompositions we refer the reader to the papers \cite{rw,rw2}. At this point we only
mention that the finite dimensional instance has found a complete solution in terms of canonical forms, see
\cite{mmx,rw2}.

\section{On the pair of operators $AB$ and $BA$ in Banach spaces}
We start this section by recalling some definitions and notions concerning the spectrum of a linear operator.
Let $\calX$ be a Banach space. The algebra of all bounded linear operators $T : \calX\to \calX$ will be
denoted by $L(\calX)$. Let $T$ be a linear operator in $\calX$ with domain $\dom T\subset \calX$. By
$\rho(T)$ we denote the {\it resolvent set of $T$} which is the set of all points $\la\in\C$ for which the
operator $T - \la : \dom T\to \calX$ is bijective and  $(T - \la)^{-1}\in L(\calX)$. Note that according to
this definition of $\rho(T)$ the operator $T$ is closed if its resolvent set is non-empty. The {\it spectrum
of $T$} is defined by $\sigma(T) := \C\setminus\rho(T)$. We define the {\it point spectrum} $\sigma_{\p}(T)$
as the set of eigenvalues of $T$.

Throughout this section we assume that $\calX$ and $\calY$ are Banach spaces, $A$ is a closed and densely
defined operator acting from $\dom A\subset\calX$ to $\calY$, and  $B$ is a closed and densely defined
operator acting from $\dom B\subset\calY$ to $\calX$. Note that the following lemma is based on linear
algebra only.

\begin{lem}\label{l:point_AB}
For $n\in\N$ and $\la\in\C\setminus\{0\}$ the operator $A$ maps $\ker((BA - \la)^n)$ bijectively onto
$\ker((AB - \la)^n)$. In particular, we have
$$
\sigma_{\p}(BA)\setminus\{0\} = \sigma_{\p}(AB)\setminus\{0\}.
$$
\end{lem}
\begin{proof}
We prove   the statement for $n=1$ only, the case of arbitrary $n$ follows by induction. Let
$\la\in\C\setminus\{0\}$ and let  $x\in\ker(BA - \la)$. Then from $BAx = \la x$ we conclude that $BAx\in\dom
A$ and $ABAx = \la Ax$. Hence, $Ax\in\ker(AB - \la)$. It is now easy to check, that $\lambda^{-1}B|\ker(AB -
\la)$ is the  inverse of $A|\ker(BA - \la)$.
\end{proof}

Observe the following simple consequence of the closed graph theorem.

\begin{lem}\label{domin}
If the operators $B$ and $AB$ are closed then there exists $c > 0$ such that
$$
\|Bx\|\leq c(\| ABx \| +\|x\|),\quad x\in\dom(AB).
$$
\end{lem}

The above phenomenon is called domination (of $B$ by $AB$). See \cite{css,ss,w} and the literature quoted
therein for an extensive research on domination in Hilbert spaces.

The first statement of the theorem below (equality \eqref{HMM}) has already been proved by V.\ Hardt and R.\
Mennicken in \cite{hm} with the use of an operator matrix construction. Nevertheless, we present a simpler
proof. It is worth mentioning that the formulas for $(BA-\lambda)^{-1}$ agree with those for the bounded
case, to be found e.g. in \cite[Problem 76]{h}. The estimate \eqref{e:res_estimate_AB} plays an important
role in the second part of the paper.

\begin{thm}\label{t:rho_AB}
Assume that the resolvent sets $\rho(AB)$ and $\rho(BA)$ of the operators $AB$ and $BA$ are non-empty. Then
we have
\begin{equation}\label{HMM}
\sigma(AB)\setminus\{0\} = \sigma(BA)\setminus\{0\}.
\end{equation}
Moreover, for $\la\in\rho(AB)\setminus\{0\}$ and $\mu\in\rho(BA)$ the following connection between the
resolvents of $AB$ and $BA$ holds:
\begin{align}\label{ppp}
(BA - \la)^{-1} =& \la^{-1} [\overline{B (AB - \la)^{-1} A} - I]\\
 = &\la^{-1}\left( \mu + (\la - \mu)B(AB - \la)^{-1}A \right)(BA - \mu)^{-1}.\label{resolvents}
\end{align}
Consequently, there exists a constant $C > 0$, which  depends on $A$ and $B$ only, such that for
$\la,\mu\in\rho(BA)$, $\la\neq 0$, the following inequality is satisfied
\begin{equation}\label{e:res_estimate_AB}
\|(BA - \la)^{-1}\| \le \frac{CM_1(\lambda)M_2(\mu)}{|\la|}\,\Big( |\mu| + |\la -
\mu|\,(2+|\la|)(2+|\mu|)\Big),
\end{equation}
with $M_1(\lambda) := \max\{1,\|(AB - \la)^{-1}\|\}$ and $M_2(\mu) := \max\{1,\|(BA - \mu)^{-1}\|\}$.
\end{thm}
\begin{proof}
Let $\la\in\rho(AB)\setminus\{0\}$. By $(BA - \la)^{-1}$ we denote the inverse of $BA - \la$ which exists due
to Lemma \ref{l:point_AB} and maps $\ran(BA - \la)$ bijectively onto $\dom(BA)$. Consider the operator $R_\la
: \dom A\to\calX$, defined by
$$
R_\la x := \la^{-1} [B (AB - \la)^{-1} Ax - x],\quad x\in\dom A.
$$
For $x\in\dom A$ we obtain
$$
A R_\la x = \la^{-1} [(AB-\la+\la) (AB - \la)^{-1} Ax - Ax]=(AB - \la)^{-1} Ax\in\dom B
$$
and hence
$$
(BA - \la)\,R_\la x = x.
$$
In particular,
\begin{equation}\label{e:dom_in_ran_AB}
\dom A\,\subset\,\ran(BA - \la)
\end{equation}
and
\begin{equation}\label{e:Rla_AB}
R_\la = (BA - \la)^{-1}|\dom A.
\end{equation}
Choose $\mu$ in the resolvent set of $BA$. Using \eqref{e:dom_in_ran_AB} and \eqref{e:Rla_AB} we obtain for
$x\in\ran(BA - \la)$
\begin{align}
\begin{split}\label{e:res_by_res}
(BA - \la)^{-1}x
&= (BA - \mu)^{-1}x + (\la - \mu)R_\la\,(BA - \mu)^{-1}x\\
&= \mu\la^{-1}(BA - \mu)^{-1}x\\
&\hspace{1cm}+ (\la - \mu)\la^{-1}\,[B(AB - \la)^{-1}]\,[A(BA - \mu)^{-1}]x.
\end{split}
\end{align}
Since the operators $B(AB - \la)^{-1}$ and $A(BA - \mu)^{-1}$ are bounded due to the closed graph theorem,
$(BA - \la)^{-1}$ is bounded as well. Since it is also closed and densely defined by \eqref{e:dom_in_ran_AB},
we obtain $\la\in\rho(BA)$, which proves \eqref{HMM}. The formulas  \eqref{ppp} and \eqref{resolvents} now
follow from \eqref{e:Rla_AB} and \eqref{e:res_by_res}, respectively.

Observe that by Lemma \ref{domin} and the triangle inequality,
$$
\|B(AB - \la)^{-1}\| \le c_1\big( \|(AB - \la)^{-1}\| + \|AB(AB - \la)^{-1}\| \big) \le c_1 M_1(\lambda)
(2+|\la|).
$$
Interchanging the roles of $A$ and $B$ we obtain for  $\mu\in\rho(BA)$
$$
\|A(BA - \mu)^{-1}\| \le c_2 M_2(\mu) (2 + |\mu|).
$$
Now, it is easy to see that these estimates, together with \eqref{resolvents}, imply
\eqref{e:res_estimate_AB} with $C := \max\{1,c_1c_2\}$.
\end{proof}

For a proof of the following proposition see \cite[Remark 2.5]{hkm} and \cite[Corollary 1.7]{hm}.

\begin{prop}\label{r:AB_star}
Let $\rho(AB)$ and $\rho(BA)$ be non-empty. Then $AB$ and $BA$ are densely defined. Moreover, we have
$$
(AB)' = B'A'\quad\text{ and }\quad(BA)' = A'B',
$$
where $'$ denotes the Banach space adjoint of densely defined linear operators in $\calX$ or in $\calY$ or
between these spaces. In consequence, if $\calX$ and $\calY$ are Hilbert spaces then
$$
(AB)^* = B^*A^*\quad\text{ and }\quad(BA)^* = A^*B^*.
$$
\end{prop}

Let $T$ be a closed and densely defined linear operator in a Banach space $\mathcal{X}$. The {\it
approximative point spectrum} $\sap(T)$ of $T$ is the set of all complex numbers $\la$ for which there exists
a sequence $(x_n)_{n=0}^\infty\subset\dom T$ with $\|x_n\| = 1$  and $(T - \la)x_n \to 0$ as $n\to\infty$.
Obviously, $\sap(T)$ is a subset of the spectrum of $T$. Note that
$$
\la\notin\sap(T)\quad\Llra\quad\ker(T-\la) = \{0\}\;\text{ and }\;\ran(T-\la)\,\text{ is closed}.
$$
The {\it continuous spectrum} $\sigma_{\cc}(T)$ and the {\it residual spectrum } $\sigma_{\textrm{r}}(T)$ of
$T$ are defined as usual. The operator $T$ is called {\it upper} ({\it lower}) {\it semi-Fredholm} if
$\ran\,T$ is closed and $\ker\,T$ is finite-dimensional (resp. $\ran\,T$ is finite-codimensional). The
operator $T$ is called {\it Fredholm} if it is both upper and lower semi-Fredholm. Note that $T$ is upper
(lower) semi-Fredholm if and only if $T'$ is lower (resp.\ upper) semi-Fredholm. The {\it essential spectrum
of $T$} is defined by
$$
\sigma_{\ess}(T) := \{\la\in\C : \,T - \la\;\text{ is not Fredholm}\,\}.
$$

\begin{thm}\label{p:spectra_AB}
Let $\rho(AB)$ and $\rho(BA)$ be non-empty. Then for $\la\in\C\setminus\{0\}$ the following statements hold:
\begin{itemize}
\item[(i)]   $\ran(AB - \la)$ is closed if and only if $\ran(BA - \la)$ is closed;
\item[(ii)]  $\ran(AB - \la)$ is dense in $\calY$ if and only if $\ran(BA - \la)$ is dense in $\calX$;
\item[(iii)] $AB-\lambda$ is upper  semi-Fredholm if and only if $BA-\lambda$ is upper semi-Fredholm;
\item[(iv)]  $AB-\la$ is lower semi-Fredholm if and only if $BA-\lambda$ is lower semi-Fredholm.
\end{itemize}
In consequence,
$$\begin{array}{cc}
\sap(AB)\setminus\{0\} = \sap(BA)\setminus\{0\}, &\sigma_{\cc}(AB)\setminus\{0\} =
\sigma_{\cc}(BA)\setminus\{0\},\\
\sigma_{\rr}(AB)\setminus\{0\} = \sigma_{\rr}(BA)\setminus\{0\}, &\sigma_{\ess}(AB)\setminus\{0\} =
\sigma_{\ess}(BA)\setminus\{0\}
\end{array}.
$$
\end{thm}
\begin{proof}
Obviously, it is sufficient to prove only one of the implications in each of the points (i)--(iv).

(i) Assume that $\ran(BA - \la)$ is not closed. Then there exists a sequence
$(x_n)_{n=0}^\infty\subset\dom(BA)$ with\footnote{Indeed, consider the quotient Banach space $X/\ker(BA-\la)$
and the injective operator $C$ that maps the equivalence class $f+\ker(BA-\la)$ ($f\in\dom(BA)$) to
$(BA-\la)f$. Then, the range of $BA - \la$ coincides with the range of $C$ and the latter is closed if and
only if $C$ is bounded from below.}
\begin{equation}\label{e:dist}
\dist\left(x_n,\ker(BA - \la)\right) = 1 \quad\text{ and }\quad (BA - \la)x_n\to 0\;\text{ as }\;n\to\infty.
\end{equation}
Fix $\mu\in\rho(BA)\setminus\{0,\la\}$  and set
$$
y_n := (\la - \mu)(BA - \mu)^{-1}x_n.
$$
Then $y_n\in\dom((BA)^2)$ for every $n\in\N$ and
\begin{equation}\label{e:xnyn_AB}
(BA - \la)y_n = (\la - \mu)\big( x_n + (\mu - \la)(BA - \mu)^{-1}x_n \big) = (\la - \mu)(x_n - y_n).
\end{equation}
On the other hand, we have
\begin{equation}\label{two-seven}
(BA - \la)y_n = (\la - \mu)(BA - \mu)^{-1}(BA - \la)x_n\;\to\,0\ \textrm{ as }\ n\to\infty.
\end{equation}
Consequently, $\|x_n - y_n\|\to 0$ as $n\to\infty$. Furthermore,  \eqref{e:xnyn_AB} gives
\begin{align*}
BA(BA - \la)y_n &= (\la - \mu)BA(x_n - y_n),
%&= (\la - \mu)\big[ (BA - \la)(x_n - y_n) + \la(x_n - y_n) \big],
\end{align*}
which also tends to zero as $n\to\infty$, by \eqref{e:dist} and \eqref{two-seven}. By Lemma \ref{domin} we
have
\begin{equation}\label{doublecross}
(AB - \la)Ay_n = A(BA - \la)y_n\;\to\,0\quad\text{as }\,n\to\infty.
\end{equation}
Now we show that
\begin{equation}\label{star}
\liminf_{n\to\infty}\dist\left(Ay_n,\ker(AB - \la)\right) > 0,
\end{equation}
which will prove that  $\ran(AB - \la)$ is not closed. Let us suppose that \eqref{star} is not true. Without
loss of generality we can assume that
 \begin{equation}\label{spades} \dist\left(Ay_n,\ker(AB -
\la)\right)\to 0\quad\text{ as }\quad n\to\infty.
\end{equation}
From \eqref{e:xnyn_AB} and \eqref{doublecross} we obtain
$$
Ax_n - Ay_n = \frac 1 {\la - \mu}\,(AB - \la)Ay_n\to 0,
$$
and consequently (cf. \eqref{spades}) $\dist\left(Ax_n,\ker(AB - \la)\right)\to 0$ as $n\to\infty$. This
implies that there exists a sequence $(u_n)_{n=0}^\infty\subset\ker(AB - \la)$ with $\|Ax_n - u_n\|\to 0$ as
$n\to\infty$. Since
$$
(BA - \la)B(AB - \mu)^{-1}u_n = B(AB - \la)(AB - \mu)^{-1}u_n = 0,
$$
we have $B(AB-\mu)^{-1}u_n\in\ker(BA-\la)$. As  $B(AB-\mu)^{-1}$ is bounded we get
$$
\dist\left(B(AB - \mu)^{-1}Ax_n,\ker(BA - \la)\right)\to 0\quad\text{with }\,n\to\infty.
$$
In view of
\begin{align*}
B(AB - \mu)^{-1}Ax_n &= BA(BA - \mu)^{-1}x_n=\\
 x_n + \mu(BA - \mu)^{-1}x_n &= x_n - y_n + \frac\la{\la - \mu}\,y_n
%
%x_n - y_n + \frac\la{\la - \mu}\,y_n
%&= x_n + \mu(BA - \mu)^{-1}x_n\\
%&= BA(BA - \mu)^{-1}x_n\\
%&= B(AB - \mu)^{-1}Ax_n
\end{align*}
together with $\|x_n - y_n\|\to 0$ as $n\to\infty$   we conclude that
 $$
\dist\left(x_n,\ker(BA - \la)\right)\  \to\  0\ \text{  as }\ n\to\infty,
$$
which is a contradiction to \eqref{e:dist}.

(ii) Let  $\ran(AB - \la)$ be dense in $\calY$ and let $x\in\dom A$ be arbitrary. We will show that
$x\in\overline{\ran(BA-\la)}$, which will finish the proof of (ii). By assumption there exists a sequence
$(v_n)_{n=0}^\infty\subset\dom(AB)$ such that $(AB - \la)v_n\to Ax$ as $n\to\infty$. Fix
$\mu\in\rho(BA)\setminus\{0\}$. We  show now that  $(BA - \la)u_n\to x$ as $n\to\infty$ where
$$
u_n := \la^{-1}(BA - \mu)^{-1}\big( (\la - \mu)Bv_n + \mu x \big) \in \dom(BA).
$$
To obtain this, observe first that for every $u\in\dom A$ we have
$$
(BA - \mu)^{-1}u = R_\mu u := \mu^{-1}\big( B(AB - \mu)^{-1}Au - u \big)
$$
(see  Theorem \ref{t:rho_AB}). Thus
\begin{align*}
(BA - \la)(BA - \mu)^{-1}u
&= BA(BA - \mu)^{-1}u - \la(BA - \mu)^{-1}u\\
&= B(AB - \mu)^{-1}Au - \frac{\la}{\mu}\left(B(AB - \mu)^{-1}Au - u\right)\\
&= \frac{1}{\mu}\left((\mu - \la)B(AB - \mu)^{-1}Au + \la u\right).
\end{align*}
Substituting $u:=u_n$ ($n\in\N$)  above and using the fact that $R_\mu B \subset B(AB - \mu)^{-1}$ we obtain
\begin{align*}
(BA \ &-\  \la)u_n
= \frac{1}{\la\mu}\left((\mu - \la)B(AB - \mu)^{-1}A + \la\right)\big( (\la - \mu)Bv_n + \mu x \big)\\
=& \ x + \frac{\la - \mu}{\la\mu}\big( (\mu - \la)B(AB - \mu)^{-1}ABv_n + \la Bv_n- \mu B(AB - \mu)^{-1}Ax \big)\\
=& \ x + \frac{\la - \mu}{\la\mu}\big( \mu B(AB - \mu)^{-1}(ABv_n - Ax) - \la\mu R_\mu Bv_n \big)\\
=& \ x + \frac{\la - \mu}{\la}B(AB - \mu)^{-1}\big( (AB - \la)v_n - Ax \big),
\end{align*}
which tends to $x$ as $n\to\infty$, since $B(AB - \mu)^{-1}$ is bounded.

Point (iii) is an easy consequence of (i) and  Lemma \ref{l:point_AB}. To see that (iv) holds suppose that
$AB - \la$ is lower semi-Fredhom. Then $(AB - \la)'$  is upper semi-Fredholm.  On the other hand  the latter
operator equals  $B'A' - \la$, cf. Proposition \ref{r:AB_star}. Since $\rho(B'A') = \rho(AB)\neq\emptyset$
and $\rho(A'B') = \rho(BA)\neq\emptyset$ we conclude that $A'B' - \la$ is upper semi-Fredholm. Consequently,
$BA - \la$ is lower semi-Fredholm. The remainder of the theorem follows directly from (i)--(iv) and Lemma
\ref{l:point_AB}.
\end{proof}

\section{Local spectral properties of $T^\adj T$ and $TT^\adj$}

For an introduction to Krein spaces and operators acting therein we refer to the monographs \cite{ai} and
\cite{b} and also to \cite{l}. Throughout this section  $(\calK,\product)$ will be a Krein space and
$\|\cdot\|$ will be a Banach space norm on $\calK$, such that the indefinite inner product is continuous with
respect $\|\cdot\|$. All such norms are equivalent and the calculations below do not depend on the choice of
one of these norms.

In what follows $T$ stands for a closed, densely defined linear operator in $\calK$. The adjoint of $T$ with
respect to $\product$ will be denoted by $T^\adj$. Observe that if $T^\adj T\in L(\calK)$ then $T\in
L(\calK)$ as well, by the closed graph theorem. Let us also note that the operator $T^\adj T$ is symmetric,
although not necessarily densely defined, cf.\ \cite[Section 3]{rw}. This was a reason for introducing in
\cite{rw} the additional assumptions (t1)--(t3), quoted in the introduction, on the operator $T$. It turns
out that assuming (t1) is not necessary.

\begin{thm}\label{t:rho->sa}
If $T$ satisfies  {\rm (t2)} then it satisfies  {\rm (t1)} as well.
\end{thm}
\begin{proof}
Note that if the resolvent sets of both $T^\adj T$ and $TT^\adj$ are non-empty, then the domain of $T^\adj T$
is dense in $\calK$, by Proposition \ref{r:AB_star}. Let $J$ be any fundamental symmetry of the Krein space
and let $A^*$ denote the adjoint of a densely defined operator $A$ in the Hilbert space
$(\calK,[J\cdot,\cdot])$. Then from $A^\adj = JA^*J$ and Proposition \ref{r:AB_star} it follows that
$$
(T^\adj T)^\adj = J(T^\adj T)^*J = JT^*(T^\adj)^*J = \big(JT^*J\big)\big(J(T^\adj)^*J\big) = T^\adj T.
$$
Similarly,  $(TT^\adj)^\adj = TT^\adj.$
\end{proof}

Recall that  a well known sufficient condition for selfadjointness of a symmetric operator in a Krein space
is that both $\la$ and $\ol\la$ belong to its resolvent set for some $\lambda\in\Comp$.  Theorem
\ref{t:rho->sa}  provides another  sufficient condition for selfadjointness of  $T^\adj T$.

%\begin{rem}
 %theorem above gives another sufficient condition for the operator of the form $T^\adj T$:

%It is well known that a symmetric operator in a Krein space is selfadjoint if only two points $\la$ and $\ol\la$ belong to its resolvent set. Theorem \ref{t:rho->sa} %states that for operators in $\calK\times\calK$ of the type
%$$
%A = \mat{T^\adj T}{0}{0}{TT^\adj}
%$$
%only one point in $\rho(A)$ is sufficient to ensure selfadjointness.
%\end{rem}

We formulate Theorem \ref{t:rho_AB} explicitly for the operators $T$ and $T^\adj$ as a separate result.

\begin{thm}\label{t:res}
Assume that {\rm (t2)} holds. Then we have
$$
\sigma(T^\adj T)\setminus\{0\} = \sigma(TT^\adj)\setminus\{0\},
$$
and there exists a constant $C > 0$ depending on $T$ only, such that for $\la,\mu\in\rho(T^\adj T)$, $\la\neq
0$, the following inequality holds
\begin{equation}\label{e:res_estimate}
\|(T^\adj T - \la)^{-1}\| \le \frac{CM_1(\lambda)M_2(\mu)}{|\la|}\,\Big( |\mu| + |\la -
\mu|\,(2+|\la|)(2+|\mu|)\Big),
\end{equation}
where $M_1(\lambda) := \max\{1,\|(TT^\adj - \la)^{-1}\|\}$ and $M_2(\mu) := \max\{1,\|(T^\adj T -
\mu)^{-1}\|\}$.
\end{thm}

\begin{rem}
It is not clear whether the condition (t1), weaker then (t2),  implies that the non-zero spectra of $T^\adj
T$ and $TT^\adj$ coincide. In view of Theorem \ref{t:res} we can formulate this ques\-tion as  the following
open problem:
\begin{quote}
Is it possible that (t1) holds and  $\rho(T^\adj T) = \emptyset$, while $\rho(TT^\adj)\neq\emptyset$?
\end{quote}
\end{rem}

\begin{cor}\label{c:small}
Assume that {\rm (t2)} is satisfied and that zero belongs to $\rho(T^\adj T)\cap\sigma(TT^\adj)$. Then zero
is a pole of order one of the resolvent of $TT^\adj$.

\end{cor}
\begin{proof}
It follows from Theorem \ref{t:res} that zero is an isolated spectral point of $TT^\adj$ and therefore an
isolated singularity of the resolvent of $TT^\adj$. Applying the estimate \eqref{e:res_estimate} for $\la$ in
a  deleted neighborhood of zero we see that it is a pole of order one.
\end{proof}

Let us recall the definition of the local sign type spectra of a selfadjoint operator in a Krein space.

\begin{defn}\label{d:spp}
Let $A$ be a selfadjoint operator in the Krein space $(\calK,\product)$. A point $\la\in\sigma(A)$ is called
a spectral point of {\em positive} {\rm (}{\em negative}{\rm )} {\em type of $A$} if $\la\in\sap(A)$ and if
for every sequence $(x_n)_{n=0}^\infty\subset\dom A$ with $\|x_n\| = 1$ and $(A -\la)x_n\to 0$ as
$n\to\infty$ we have
$$
\liminf_{n\to\infty}\, [x_n,x_n] > 0 \;\;\;\;
 \bigl(\,{\rm resp.\,}\, \limsup_{n\to\infty}\, [x_n,x_n] < 0 \bigr).
$$
We denote the set of all spectral points of positive {\rm (}negative{\rm )} type of $A$ by $\spp(A)$ {\rm
(}resp.\ $\smm(A)${\rm )}. A set $\Delta\subset\C$ is said to be {\it of positive {\rm (}negative{\rm )} type
with respect to $A$} if $\Delta\cap\sigma(A)\subset\spp(A)$ (resp.\ $\Delta\cap\sigma(A)\subset\smm(A)$). If
$\Delta$ is either of positive type or of negative type with respect to $A$, then we say that $\Delta$ is
{\it of definite type with respect to $A$}.
\end{defn}

It is well known that for a selfadjoint operator $A$ we have $\sigma(A)\cap\R\subset\sap(A)$. It was shown in
\cite{ajt,lmm} that $\sigma_{\pm\pm}(A)\subset\R$ and that for a closed interval $\Delta$ which is of
positive type with respect to $A$ there exists an open neighborhood $\calU$ in $\C$ of $\Delta$ such that
\begin{equation}\label{e:pppip_nbh}
\calU\cap\sigma(A)\cap\R\subset\spp(A) \quad\text{ and }\quad \calU\setminus\R\subset\rho(A).
\end{equation}
It was also shown that the resolvent $(A - \la)^{-1}$ for $\la$ near $\Delta$ does not grow faster than
$M/|\Im\la|$ with some constant $M > 0$. This fact gives rise to a local spectral function of $A$ on
$\Delta$. The spectral subspaces given by this spectral function are then Hilbert spaces with respect to the
inner product $\product$, cf. \cite{lmm}. An analogue holds for intervals of negative type with respect to
$A$.

In \cite{ajt} another type of spectral points of a selfadjoint operator in a Krein space was introduced,
namely the spectral points of type $\pi_+$ and $\pi_-$. The definition below is equivalent to that in
\cite{ajt}, see \cite[Theorem 14]{ajt}. We write $x_n\wto x$ as $n\to\infty$ if the sequence
$(x_n)_{n=0}^\infty\subset\calK$ converges weakly to some $x\in\calK$.

\begin{defn}\label{d:spip}
Let $A$ be a selfadjoint operator in the Krein space $(\calK,\product)$. A point $\la\in\sigma(A)$ is called
a spectral point {\em of type $\pi_+$} ({\it type $\pi_-$}) {\em of $A$} if $\la\in\sap(A)$ and if for every
sequence $(x_n)_{n=0}^\infty\subset\dom A$ with $\|x_n\| = 1$, $x_n\wto 0$ and $(A -\la)x_n\to 0$ as
$n\to\infty$ we have
$$
\liminf_{n\to\infty}\, [x_n,x_n] > 0 \;\;\;\;
 \bigl(\,{\rm resp.\,}\, \limsup_{n\to\infty}\, [x_n,x_n] < 0 \bigr).
$$
We denote the set of all spectral points of type $\pi_+$ (type $\pi_-$) of $A$ by $\spip(A)$ (resp.\
$\spim(A)$). A set $\Delta\subset\C$ is said to be {\it of type $\pi_+$ {\rm (}resp.\ type $\pi_-${\rm )}
with respect to $A$} if $\Delta\cap\sigma(A)\subset\spip(A)$ (resp.\ $\Delta\cap\sigma(A)\subset\spim(A)$).
\end{defn}

It was shown in \cite{bpt} (for a weaker statement see also \cite{ajt}) that if $\Delta$ is a closed interval
of type $\pi_+$ with respect to the selfadjoint operator $A$ in $\calK$ which contains an accumulation point
of the resolvent set of $A$, then -- just as in the case of an interval of definite type -- there exists a
neighborhood $\calU$ of $\Delta$ in $\C$ such that \eqref{e:pppip_nbh} holds with $\spp(A)$ replaced by
$\spip(A)$. Moreover, the set $\Delta\cap (\spip(A)\setminus\spp(A))$ is finite. The growth of $(A -
\la)^{-1}$ in $\calU\setminus\R$ can be estimated by some power of $|\Im\la|^{-1}$. Also in this case $A$
possesses a local spectral function on $\Delta$. The spectral subspaces are then Pontryagin spaces with
finite rank of negativity.

The following result generalizes Proposition 5.1 of \cite{rw}. By $\R^\pm$ we denote $\{x\in\R:\pm x>0\}$.

\begin{prop}\label{p:spectra}
Assume that {\rm (t2)} is satisfied. Then the following holds:
\begin{itemize}
\item[{\rm (i)}]   $\sap(T^\adj T)\setminus\{0\} = \sap(TT^\adj)\setminus\{0\}$,
\item[{\rm (ii)}]  $\sigma_{\pm\pm}(T^\adj T)\cap\R^+ = \sigma_{\pm\pm}(TT^\adj)\cap\R^+$,
\item[{\rm (iii)}] $\sigma_{\pm\pm}(T^\adj T)\cap\R^- = \sigma_{\mp\mp}(TT^\adj)\cap\R^-$,
\item[{\rm (iv)}]  $\sigma_{\pi_\pm}(T^\adj T)\cap\R^+ = \sigma_{\pi_\pm}(TT^\adj)\cap\R^+$,
\item[{\rm (v)}]   $\sigma_{\pi_\pm}(T^\adj T)\cap\R^- = \sigma_{\pi_\mp}(TT^\adj)\cap\R^-$.
\end{itemize}
\end{prop}
\begin{proof}
Statement (i) is a direct consequence of Theorem \ref{p:spectra_AB}. To prove (ii) and (iii) consider
$\la\in\spp(TT^\adj)\setminus\{0\}$. Then $\la\in\sap(T^\adj T)$ by (i). Let $(x_n)_{n=0}^\infty
\subset\dom(T^\adj T)$ be a sequence with $\|x_n\| = 1$ and $(T^\adj T - \la)x_n\to 0$ as $n\to\infty$. Then
define
$$
y_n := (\la - \mu)(T^\adj T - \mu)^{-1}x_n \,\in\,\dom((T^\adj T)^2)
$$
with some $\mu\in\rho(T^\adj T)\setminus\{0\}$. This sequence satisfies $\liminf_{n\to\infty}\,\|Ty_n\| > 0$
and
$$
\lim_{n\to\infty}\,\|y_n - x_n\| = \lim_{n\to\infty}\,\|(T^\adj T - \la)y_n\| = \lim_{n\to\infty}\,\|(TT^\adj
- \la)Ty_n\| = 0.
$$
This can be seen with a very similar argumentation as in the proof of Theorem \ref{p:spectra_AB} (i) (with $A
= T$ and $B = T^\adj$). Hence,
\begin{align*}
\liminf_{n\to\infty}\,[x_n,x_n]
&= \liminf_{n\to\infty}\,[y_n,y_n]\\
&= \frac 1 \la \, \liminf_{n\to\infty}\,\Big([Ty_n,Ty_n] - [(T^\adj T - \la)y_n,y_n]\Big)\\
&= \frac 1 \la \, \liminf_{n\to\infty}\,[Ty_n,Ty_n],
\end{align*}
and similarly
$$
\limsup_{n\to\infty}\,[x_n,x_n] = \frac 1 \la \, \limsup_{n\to\infty}\,[Ty_n,Ty_n].
$$
Since $\la\in\spp(TT^\adj)$, we have
$$
 \limsup_{n\to\infty}\,[Ty_n,Ty_n] \ge \liminf_{n\to\infty}\,[Ty_n,Ty_n] > 0.
$$
This shows that $\la\in\spp(T^\adj T)$ if $\la > 0$ and $\la\in\smm(T^\adj T)$ if $\la < 0$.

To show that (iv) and (v) hold, let $\la\in\spip(TT^\adj)\setminus\{0\}$. Then the same argument as above
applies with the additional assumption that $(x_n)_{n=0}^\infty$ converges weakly to zero. It remains to show
that $(Ty_n)_{n=0}^\infty$ (or at least a subsequence) converges weakly to zero. Since $T(T^\adj T -
\mu)^{-1}$ is bounded, $(Ty_n)_{n=0}^\infty$ is bounded. It is therefore no restriction to assume that there
exists some $v\in\calK$ such that $Ty_n\wto v$ as $n\to\infty$. From $x_n\wto 0$ and $\|y_n - x_n\|\to 0$ we
conclude that $y_n\wto 0$ as $n\to\infty$. Since $T$ is (weakly) closed, $v$ equals zero.
\end{proof}

The definition of local definitizability was first formulated in  1986 in  \cite[Section 2.2]{j86} for
unitary operators. The version below  is taken from \cite{j03}, see also \cite{j88}. We denote the one-point
compactification of the real line and the complex plane by $\ol\R$ and $\ol\C$, respectively.

\begin{defn}\label{d:locdef}
Let $\Omega$ be a domain in $\ol\C$ which is symmetric with respect to $\ol\R$ with
$\Omega\cap\R\neq\emptyset$ such that $\Omega\cap\C^+$ and $\Omega\cap\C^-$ are simply connected. A
selfadjoint operator $A$ in $\calK$ is called {\em definitizable over \,$\Omega$} if the following holds:
\begin{itemize}
\item[{\rm (i)}]   The set $\sigma(A)\cap(\Omega\setminus\ol\R)$ does not have any accumulation points in $\Omega$ and consists of poles of the resolvent of $A$.
\item[{\rm (ii)}]  For each closed subset $\Delta$ of $\Omega\cap\ol\R$ there exist an open neighborhood $\calU$ of $\Delta$ in $\ol\C$ and numbers $m \geq 1$, $M > 0$ such that
$$
\|(A - \lambda)^{-1}\| \leq M\,\frac{(1 + |\la|)^{2m-2}}{|\mbox{{\rm Im}}\, \lambda |^{m}},\qquad
\lambda\in\calU\setminus\ol\R.
$$
\item[{\rm (iii)}] Each point $\lambda\in\Omega\cap\ol\R$ has an open connected neighborhood $I_{\lambda}$ in $\ol\R$ such that each component of $I_\lambda\setminus\{\la\}$ is of definite type with respect to $A$.
\end{itemize}
\end{defn}

%
%Above, it was mentioned that if $\Delta$ is a closed interval of type $\pi_+$ or of positive type with respect to $A$, then there exists a local spectral function for $A$ on $\Delta$. In this case the operator is definitizable over some $\Omega$ as in Definition \ref{d:locdef} with $\Delta\subset\Omega$, see \cite[Theorem 23]{ajt}. Moreover, the spectral points of type $\pi_+$ and of positive type of $A$ in $\Delta$ can be retrieved in terms of the signatures of the spectral subspaces defined by the spectral function.

In \cite[Theorem 3.6]{j03} it was proved that a selfadjoint operator $A$ in the Krein space $\calK$ is
definitizable if and only if it is definitizable over $\ol\C$. The following theorem was proved in \cite{rw}
for the special case $\Omega = \ol\C$.

\begin{thm}\label{t:locdef}
Assume that {\rm (t2)} holds and let $\Omega$ be an open domain in $\ol\C$ as in Definition {\rm
\ref{d:locdef}}. Then $T^\adj T$ is definitizable over $\Omega$ if and only if $TT^\adj$ is definitizable
over $\Omega$.
\end{thm}
\begin{proof}

Let us assume that $TT^\adj$ is definitizable over $\Omega$. By Theorem \ref{t:res} and Proposition
\ref{p:spectra} the conditions (i) and (iii) in Definition \ref{d:locdef} are easily seen to be satisfied by
$T^\adj T$. Hence, it remains to check that condition (ii) holds for $T^\adj T$. Let $\Delta$ be a closed
subset of $\Omega\cap\ol\R$. Then, as $TT^\adj$ is definitizable over $\Omega$, there exist an open
neighborhood $\calU$ of $\Delta$ in $\ol\C$ and numbers $m\geq 1$, $M > 0$ such that
\begin{equation}\label{e:finite_growth}
\|(TT^\adj - \lambda)^{-1}\| \leq M\,\frac{(1 + |\la|)^{2m-2}}{|\Im\la|^m}
\end{equation}
holds for all $\lambda\in\calU\setminus\ol\R$. It is obviously no restriction to assume $M\ge 1$. Moreover,
as the sequence $(1 + |\la|)^{2n-2}/|\Im\la|^n$ is monotonically increasing for each $\la\in\C\setminus\R$,
we may assume $m\ge 2$, such that the right hand side of \eqref{e:finite_growth} is not smaller than $1$.

Fix some $\mu\in\rho(T^\adj T)\setminus\{0\}$. Then, by Theorem \ref{t:res} we have for
$\la\in\calU\setminus\ol\R$:
$$
\|(T^\adj T - \la)^{-1}\| \le \frac{D}{|\la|}\,\Big( |\mu| + |\la - \mu|\,(2+|\la|)(2+|\mu|)\Big)\,M_1(\la),
$$
with
\begin{align*}
M_1(\la) = \max\{1,\|(TT^\adj - \la)^{-1}\|\} \le M\,\frac{(1 + |\la|)^{2m-2}}{|\Im\la|^m}
\end{align*}
and some $D > 0$ depending on $T$ and $\mu$ only. Hence, with $c := |\mu|$ we obtain for all
$\la\in\calU\setminus\ol\R$ that
\begin{align*}
\|(T^\adj T - \la)^{-1}\|
&\le \frac{{\rm const}}{|\la|}\,\Big(c + (2+c)(c+|\la|)\,(2+|\la|)\Big)\,\frac{(1 + |\la|)^{2m-2}}{|\Im\la|^m}\\
&\le \frac{{\rm const}}{|\la|}\,\Big(1 + (c+|\la|)\,(2+|\la|)\Big)\,\frac{(1 + |\la|)^{2m-2}}{|\Im\la|^m}\\
&\le \frac{{\rm const}}{|\la|}\,\Big(1 + 2\max\{1,c\}(1+|\la|)^2\Big)\,\frac{(1 + |\la|)^{2m-2}}{|\Im\la|^m}\\
&\le \frac{{\rm const}}{|\la|}\,(1+|\la|)^2\,\frac{(1 + |\la|)^{2m-2}}{|\Im\la|^m}\\
&\le {\rm const}\,\frac{(1 + |\la|)^{2(m+1)-2}}{|\Im\la|^{m+1}},
\end{align*}
with some ${\rm const} > 0$ which is independent of $\la$.
\end{proof}

In the following let $A$ be a selfadjoint operator in $(\calK,\product)$ which is definitizable over some
domain $\Omega$. If $A$ is unbounded and infinity belongs to $\Omega$ then we say that infinity is a spectral
point of positive (negative) type of $A$ if both components of $I_\infty\setminus\{\infty\}$ (see Definition
\ref{d:locdef}) are of positive (resp.\ negative) type.

As is well known (see e.g. \cite{j03}), the operator $A$ possesses a local spectral function $E$ on
$\Omega\cap\ol\R$. The projection $E(\Delta)$ is always a bounded selfadjoint operator in the Krein space
$\calK$ and is defined for all finite unions $\Delta$ of connected subsets of $\Omega\cap\ol\R$ the endpoints
of which are of definite type with respect to $A$. We denote this system of sets by $\calR_\Omega(A)$. The
spectral points of definite type of $A$ can be characterized with the help of $E$: a point $\la\in\R$ is a
spectral point of positive (negative) type of $A$ if and only if for some open $\Delta\in\calR_\Omega(A)$
with $\la\in\Delta$ the space $(E(\Delta)\calK,\product)$ (resp.\ $(E(\Delta)\calK,-\product)$) is a Hilbert
space (cf.\ \cite[Theorem 2.15]{j03}).

In analogy to definitizable operators we say that a point $\la\in\Omega\cap\sigma(A)\cap\R$ (or $\la=\infty$
if $A$ is unbounded) is a {\it critical point} of $A$ if it is not a spectral point of definite type of $A$.
This is obviously equivalent to the fact that for any $\Delta\in\calR_\Omega(A)$ with $\la\in\Delta$ the
space $(E(\Delta)\calK,\product)$ is indefinite. The set of critical points of $A$ in $\Omega$ will be
denoted by $c_\Omega(A)$.

The critical point $\la$ of $A$ is called {\it regular} if there exists $c > 0$ such that for some (and hence
for any) $\Delta_0\in\calR_\Omega(A)$ with $\Delta_0\cap c_\Omega(A) = \{\la\}$ we have $\|E(\Delta)\|\le c$
for all $\Delta\in\calR_\Omega(A)$, $\Delta\subset\Delta_0$. If $\la$ is not regular it is called a {\it
singular} critical point. If for any $\Delta\in\calR_\Omega(A)$ with $\lambda\in\Delta$ the space
$(E(\Delta)\calK,\product)$ is not a Pontryagin space, then $\la$ is called an {\it essential} critical point
of $A$. If $\infty\in\Omega$ is a critical point then it is always essential. In \cite[Theorem 26]{ajt} it
was shown that a critical point $\la\neq\infty$ is essential if and only if it is neither of type $\pi_+$ nor
of type $\pi_-$.

The statements (i), (ii), (iv) and (v) of the following theorem were  proved in \cite{rw} for the case that
$T^\adj T$ and $TT^\adj$ are definitizable. It turns out that these results also hold when $T^\adj T$ and
$TT^\adj$ are only definitizable over some domain $\Omega$.

\begin{thm}\label{critical}
Assume that {\rm (t2)} holds and let $T^\adj T$ {\rm (}and hence also $TT^\adj${\rm )} be definitizable over
some domain $\Omega$ as in Definition {\rm \ref{d:locdef}}. Then for $\la\in\R\setminus\{0\}$ the following
statements hold:
\begin{itemize}
\item[{\rm (i)}]   $\la$ is a critical point of $T^\adj T$ if and only if it is a critical point of
$TT^\adj$;
\item[{\rm (ii)}]  $\la$ is a regular critical point of $T^\adj T$ if and only if it is a regular critical point of
$TT^\adj$;
\item[{\rm (iii)}] $\la$ is an essential critical point of $T^\adj T$ if and only if it is an essential critical point of
$TT^\adj$;
\end{itemize}
Moreover,
\begin{itemize}
\item[{\rm (iv)}] if zero is a singular critical point of $T^\adj T$ then zero belongs to $\sigma(TT^\adj)$;
\item[{\rm (v)}]  if infinity is a critical point of $T^\adj T$ then infinity is of definite type with respect to $TT^\adj$.
\end{itemize}
\end{thm}
\begin{proof}
The assertions (i), (iii) and (v) are immediate consequences of Proposition \ref{p:spectra}. The proof of
(ii) follows analogous lines as the proof of Theorem 4.2(iii) of \cite{rw}, with the use of the local
spectral function instead of the spectral function of a definitizable operator. We leave the details to the
reader. Point (iv) results from the equality of non-zero spectra of $T^\adj T$ and $TT^\adj$ and from the
fact that an isolated point of the spectrum cannot be a singular critical point.
\end{proof}

We conclude this paper with an example, in which the operator $T^\adj T$ is easily seen to be locally
definitizable, while $TT^\adj$ has a relatively complicated form.

\begin{ex}
Let $(\calK_0,\product)$ be an infinite--dimensional Krein space and let $\calK=\calK_0\times\calK_0$ with
the standard product indefinite metric. Let $T_0\in L(\calK_0)$ be such that $T_0^\adj T_0$ is locally
definitizable over some set $\Omega$ (see e.g. Example \ref{e:one}) and let $T_1\in L(\calK_0)$ be an
operator having a neutral range, which is  equivalent to  $T_1^\adj T_1=0$. Consider the operator
 $$
 T=\begin{pmatrix} T_0 & 0\\ T_1 & 0\end{pmatrix}\in L(\calK).
 $$
Then
 $$
 T^\adj T =\begin{pmatrix} T_0^\adj T_0 & 0\\ 0 & 0\end{pmatrix}\in L(\calK),
 $$
and it is clearly locally definitizable over $\Omega$. Although the operator $TT^\adj$ has a more complicated
form, namely
 $$
 T T ^\adj=\begin{pmatrix} T_0 T_0^\adj & T_0 T_1 ^\adj \\ T_1T_0^\adj & T_1T_1^\adj\end{pmatrix}\in L(\calK),
 $$
we know by Theorem \ref{t:locdef} that it is locally definitizable over $\Omega$ as well.

\end{ex}

\subsection*{Acknowledgment}
We would like to thank our colleagues Piotr Budzy\'nski and Carsten Trunk  for helpful discussions and
inspirations.

% ------------------------------------------------------------------------
\end{document}